\documentclass[12pt]{amsart}

\usepackage{fancybox,color, mathtools, array}

\newcommand{\kom}[1]{}

\renewcommand{\kom}[1]{{\bf [#1]}}

\definecolor{gruen}{cmyk}{1.0,0.2,0.7,0.07}

\usepackage{latexsym}
\usepackage{color}
\usepackage{amssymb}
\usepackage{mathrsfs}
\usepackage{amsmath}
\usepackage{fancybox,color,graphicx}
\usepackage{enumerate}

\usepackage[active]{srcltx}
\usepackage[colorlinks]{hyperref}
\usepackage{hypernat}
\usepackage{setspace}

\usepackage{soul}

 \newcommand{\eps}{{\varepsilon}}
 
 \def\1{\raisebox{2pt}{\rm{$\chi$}}}
 
\newcommand{\norm}[1]{\left|\left|#1\right|\right|}

\def\vint_#1{\mathchoice%
          {\mathop{\kern 0.2em\vrule width 0.6em height 0.69678ex depth -0.58065ex
                  \kern -0.8em \intop}\nolimits_{\kern -0.4em#1}}%
          {\mathop{\kern 0.1em\vrule width 0.5em height 0.69678ex depth -0.60387ex
                  \kern -0.6em \intop}\nolimits_{#1}}%
          {\mathop{\kern 0.1em\vrule width 0.5em height 0.69678ex
              depth -0.60387ex
                  \kern -0.6em \intop}\nolimits_{#1}}%
          {\mathop{\kern 0.1em\vrule width 0.5em height 0.69678ex depth -0.60387ex
                  \kern -0.6em \intop}\nolimits_{#1}}}
\def\vintslides_#1{\mathchoice%
          {\mathop{\kern 0.1em\vrule width 0.5em height 0.697ex depth -0.581ex
                  \kern -0.6em \intop}\nolimits_{\kern -0.4em#1}}%
          {\mathop{\kern 0.1em\vrule width 0.3em height 0.697ex depth -0.604ex
                  \kern -0.4em \intop}\nolimits_{#1}}%
          {\mathop{\kern 0.1em\vrule width 0.3em height 0.697ex depth -0.604ex
                  \kern -0.4em \intop}\nolimits_{#1}}%
          {\mathop{\kern 0.1em\vrule width 0.3em height 0.697ex depth -0.604ex
                  \kern -0.4em \intop}\nolimits_{#1}}}

\newcommand{\aveint}[2]{\mathchoice%
          {\mathop{\kern 0.2em\vrule width 0.6em height 0.69678ex depth -0.58065ex
                  \kern -0.8em \intop}\nolimits_{\kern -0.45em#1}^{#2}}%
          {\mathop{\kern 0.1em\vrule width 0.5em height 0.69678ex depth -0.60387ex
                  \kern -0.6em \intop}\nolimits_{#1}^{#2}}%
          {\mathop{\kern 0.1em\vrule width 0.5em height 0.69678ex depth -0.60387ex
                  \kern -0.6em \intop}\nolimits_{#1}^{#2}}%
          {\mathop{\kern 0.1em\vrule width 0.5em height 0.69678ex depth -0.60387ex
                  \kern -0.6em \intop}\nolimits_{#1}^{#2}}}

\newtheorem{theorem}{Theorem}[section]

\newtheorem{lemma}[theorem]{Lemma}
\newtheorem{proposition}[theorem]{Proposition}
\newtheorem{definition}[theorem]{Definition}

 \renewcommand{\div}{\operatorname{div}}
 
  \newcommand{\R}{\mathbb{R}}

 \newcommand{\essinf}{\operatorname{ess \, inf}}
 \newcommand{\esssup}{\operatorname{ess \, sup}}

\newcommand{\xintloo}[1]{\int\limits_{#1} \kern-18pt\raise4pt\hbox to7pt {\hrulefill}\ }   \numberwithin{equation}{section}

\newcommand{\ec}{\normalcolor}

\def\Xint#1{\mathchoice
   {\XXint\displaystyle\textstyle{#1}}%
   {\XXint\textstyle\scriptstyle{#1}}%
   {\XXint\scriptstyle\scriptscriptstyle{#1}}%
   {\XXint\scriptscriptstyle\scriptscriptstyle{#1}}%
   \!\int}
\def\XXint#1#2#3{{\setbox0=\hbox{$#1{#2#3}{\int}$}
     \vcenter{\hbox{$#2#3$}}\kern-.5\wd0}}

\def\dashint{\Xint-}

 \newcommand{\mup}{\mu_p^{\mathbb{H}}}
 \newcommand{\dz}{\,\mbox{d}z}
 \newcommand{\dy}{\,\mbox{d}y}
 \newcommand{\dt}{\,\mbox{d}t}
 
\begin{document}
 \title[Natural $p$-Means]{Convergence of 
  Natural $p$-Means \\ for the $p$-Laplacian in the Heisenberg group}
 
 \author{Andr\'{a}s Domokos}
 \address{Andr\'{a}s Domokos
 	\hfill\break\indent
 	Department of Mathematics and Statistics
 	\hfill\break\indent
 	California State University Sacramento
 	\hfill\break\indent 6000 J Street, Sacramento, CA, 95819, USA
 	\hfill\break\indent \ec
 	{\tt domokos@csus.edu}}
 
  \author[J. Manfredi]{Juan J. Manfredi}
 \address{Juan J. Manfredi
 \hfill\break\indent
Department of Mathematics
\hfill\break\indent
University of Pittsburgh
\hfill\break\indent Pittsburgh, PA 15260, USA
\hfill\break\indent
{\tt manfredi@pitt.edu}}

\author{Diego Ricciotti}
\address{Diego Ricciotti
	\hfill\break\indent
	Department of Mathematics and Statistics
	\hfill\break\indent
	California State University Sacramento
	\hfill\break\indent 6000 J Street, Sacramento, CA, 95819, USA
	\hfill\break\indent \ec
	{\tt ricciotti@csus.edu}}

  \author[B. Stroffolini]{Bianca Stroffolini}
 \address{Bianca Stroffolini
 \hfill\break\indent
 Department of Electrical Engineering and Information Technology
\hfill\break\indent
University Federico II Napoli
\hfill\break\indent Napoli, 80125, Italy
\hfill\break\indent \ec
{\tt bstroffo@unina.it}}

 \date{\today}
\maketitle
\begin{center}
\textit{Dedicated to the memory of Emmanuele DiBenedetto.} 
\end{center}
\begin{abstract} In this paper we prove uniform convergence of approximations to $p$-harmonic functions by using natural $p$-mean operators on bounded domains of the Heisenberg group $\mathbb{H}$ which satisfy an intrinsic exterior corkscrew condition. These domains include Euclidean $C^{1,1}$ domains.
 
\end{abstract}

\section{Introduction} 

Solutions  to a large class of elliptic and parabolic equations can be characterized by \linebreak asymptotic mean value properties (see for example \cite{MPR10, BCMR21}). Consider the case of $p$-harmonic functions for $1<p< \infty$. A smooth function with non-vanishing gradient satisfies at a point $x\in\mathbb{R}^n$ the $p$-harmonic equation 
$$\sum_{i,j=1}^n \left\{ \delta_{ij} + (p-2) \frac{u_{x_i}(x) u_{x_j}(x)}{|\nabla u(x)|^2} \right\} u_{x_i x_j}(x)=0,$$ if and only if 
\begin{equation}\label{expansion1}
u(x)= \frac{\alpha}2 \left(\sup_{B_{\eps}(x)} u + \inf_{B_{\eps}(x)} u\right)+\beta\, \dashint_{B_{\eps}(x)}u(y)\, dy+o(\eps^2),
\end{equation}
where $\alpha=\frac{p-2}{n+p}$ and  $\beta=\frac{n+2}{n+p}$.
For general viscosity solutions the same characterization prevails provided that we interpret \eqref{expansion1} in the viscosity sense. 
\par
The expansion \eqref{expansion1} suggests the consideration of solutions $u_\eps$ to the  Dynamic Programming Principle (DPP)
\begin{equation}\label{dpp1}
u_\eps(x)= \frac{\alpha}2 \left(\sup_{B_{\eps}(x)} u_\eps+   \inf_{B_{\eps}(x)} u_\eps \right)+\beta\, \dashint_{B_{\eps}(x)}u_\eps(y)\, dy, \end{equation} 
where $\Omega\subset\mathbb{R}^n$  is a bounded domain, the function $u_\eps\colon\Omega\to\mathbb{R}$ and $B_\eps(x)\subset \Omega$. Suppose that $g\colon\partial\Omega\to\mathbb{R}$ is continuous. In order to consider \eqref{dpp1} for all points $x\in\overline\Omega$ and for  all $0< \eps \leq 1$, consider the $\eps$-boundary strip 
$$\Gamma_{\eps} = \{x\in \mathbb{R}^n\setminus \Omega\,\colon\, \text{dist}(x, \Omega) \leq \eps\}.$$
We extend the function $g$ continuously to a function $G$ to this boundary strip and consider the problem
\begin{equation}\label{dpp2}
\left\{
\begin{array}{*3{>{\displaystyle}c}l} 
u_\eps(x) & = &  \frac{\alpha}2 \left( \sup_{B_{\eps}(x)} u_\eps +   \inf_{B_{\eps}(x)} u_\eps \right)+\beta\, \dashint_{B_{\eps}(x)}u_\eps(y)\, dy, &  x\in{\Omega},\\
u_\eps(x)&=& G(x),&  {x\in\Gamma_{\eps}}.
\end{array}
\right.
\end{equation}
Since we are requiring that $u_\eps = G$ in $\Gamma_\eps$,  the expression \eqref{dpp1} is now well defined for $x\in\overline{\Omega}$. 
When the domain $\Omega$ is Lipschitz, one can solve the Dirichlet problem \eqref{dpp2} and obtain a  family $\{u_\eps\}_{0<\eps\le 1}$ of functions such that $u_\eps\to u$ uniformly in $\overline{\Omega}$, where $u$ is the unique viscosity solution to the Dirichlet problem
\begin{equation}\label{dp1}
\left\{
\begin{array}{cccl}
\displaystyle\sum_{i,j=1}^n \left\{ \delta_{ij} + (p-2) \frac{u_{x_i} u_{x_j}}{|\nabla u|^2} \right\} u_{x_i x_j} & = & 0   , &  \text{  in  }{\Omega}  \\
u&=& g,&\text{  in   }\partial\Omega.
\end{array}
\right.
\end{equation}
We note that viscosity solutions of the problem \eqref{dp1} are also weak solutions,  when the homogeneous $p$-Laplacian is replaced by the variational
$p$-Laplacian, $\div\left(|\nabla u|^{p-2}\nabla u\right)=0$, \cite{JLM01}. For a mean value property that applies directly to the variational $p$-Laplacian see \cite{dTL21}. 
\par
The nonlinear mean value expression in the right-hand side of \eqref{dpp1} is a \textit{tug-of-war with noise} mean, suggested by connections with probability developed in  \cite{PSSW09} and  \cite{PS08}. \par

Let  $\Omega_{E}=\Omega\cup\Gamma_{1}$ and ${\mathcal B}(\overline{\Omega})$, ${\mathcal B}(\Omega_{E})$ be the class of bounded real measurable functions defined on $\overline{\Omega}$ and 
 $\Omega_{E}$, respectively.
\begin{definition}\label{def:M}
We say that an operator $A: {\mathcal B}(\Omega_E)\to {\mathcal B}(\overline{\Omega})$
is an \textit{average operator}  if it satisfies the following properties:
\begin{itemize}
\item{}(Stability) $\inf_{\{y\in \Omega_E\}} \phi(y)\le A[\phi](x)\le \sup_{\{y\in \Omega_E\}} \phi(y),\forall x\in\overline{\Omega}$;
\item{}(Monotonicity)If $\phi\le \psi$ in $\Omega_E$ then $A[\phi]\le A[\psi]$ in $\overline{\Omega}$;
\item{}(Affine invariance) $A[\lambda \phi+\xi]=\lambda A[\phi]+\xi, \forall \lambda>0, \forall \xi\in \mathbb{R}$.
\end{itemize}
\end{definition}
\begin{definition}\label{def:MVP}
We say that a family of averages $\{A_\eps\}_{\eps>0}$ satisfies the   \textbf{asymptotic mean value property (AMVP)} for the $p$-Laplacian if for every $\phi\in C^\infty(\Omega_E)$ such that $\nabla \phi \not=0$, we have
\begin{equation*} 
A_\eps[\phi](x)=\phi(x) +c\,  \eps^2\left(\Delta_p^H \phi(x)\right)+ o(\eps^2)
\end{equation*}
for some constant $c>0$ independent of $\eps$ and $\phi$,  and where the constant in $o(\eps^2)$ can be taken uniformly for all $x\in \overline{\Omega}$. 
 \end{definition}
 Associated to an average operator $A_\eps$, we have a dynamic programming principle (DPP) at scale $\eps$   given by
\begin{equation}\label{eq:expDPP1}
\left\{
\begin{array}{cccl}
u_\eps(x) & = & A_\eps [u_\eps] (x) & \text{ in } \Omega,\\
u_\eps(x)&=& G(x) & \text{ on } {\Gamma_{1}}.
\end{array}
\right.
\end{equation}
Later in Section \S \ref{22} we will discuss existence and uniqueness for the DPP \eqref{eq:expDPP1} associated to the natural $p$-means in the Heisenberg group. 

We will say that  a function $u \in {\mathcal B}(\Omega_{E})$ is a subsolution (resp. supersolution) of \eqref{eq:expDPP1} in $\Omega$ with boundary datum $G$, if $u\le  G $ (resp. $u\ge G$)  on $\Gamma_{1}$ and $u(x)\leq A_{\eps}[u] (x)$ (resp. $u(x)\geq A_{\eps}[u] (x)$) for  $x\in\Omega$.

 Consider the following conditions on the family of averages $\{A_\eps\}_{\eps>0}$:
\medskip
\par\noindent 
\textbf{Uniform Stability:}
\begin{equation}\label{as:stab}\tag{US}
\begin{split}
&\textup{For all $ \eps>0$ there exists $ u_\eps \in \mathcal{B}(\Omega_E)$, a  solution of \eqref{eq:expDPP1}
with a}\\ &\textup{bound on $\|u_\eps\|_{L^\infty(\Omega_E)}$ uniform in $\eps$.}
\end{split}
\end{equation}
\medskip
\par\noindent 
\textbf{Uniform Boundedness:}
\begin{equation}\label{as:stab2}\tag{UB}
\begin{split}
&\textup{For all $ \eps>0$ there exists $ u_\eps \in \mathcal{B}(\Omega_E)$, a  solution of \eqref{eq:expDPP1},  and }\\
&\inf_{\Gamma_1} G \leq u_\eps(x) \leq \sup_{\Gamma_1}{G} \quad \textup{for all}\quad x\in \Omega.
\end{split}
\end{equation}
\medskip
\par\noindent 
\textbf{Comparison Principle:}
\begin{equation}\label{as:comppr}\tag{CP}
\begin{split}
&\textup{Let $u^1_\eps$ and $u^2_\eps$ be a subsolution and a supersolution of \eqref{eq:expDPP1} }\\ &\textup{with boundary data $G_1$ and $G_2$ respectively.} \\ & \textup{If $G_1\leq G_2$ on $\Gamma_1$, then $u_\eps^1\leq u_\eps^2$ in $\Omega_E$.} 
\end{split}
\end{equation}
 
\begin{theorem}[Convergence of general mean approximations in the  Euclidean case, $1<p \le \infty$, \cite{dTMP20}]\label{thm1} Let 
$\Omega\subset \R^n$ be a bounded domain and $g\in C(\partial\Omega)$. Let the family of averages  $\{A_\eps\}_{\eps>0}$ satisfy the \textbf{AMVP} with respect to the $p$-Laplacian.  Let also $\{u_\eps\}_{\eps>0}$ be a sequence of solutions of the corresponding DPP \eqref{eq:expDPP1}, where $G$ is a continuous extension of $g$ to $\Gamma_1$. Then, we have that 
\[
 u_\eps\to u  \text{  uniformly  in   } \overline{\Omega}  \text{  as  } \eps\to0,  
\]
\begin{itemize}
\item when the domain $\Omega$ is of  class $C^{2}$  and the family of averages  $\{A_\eps\}_{\eps>0}$  satisfies  \eqref{as:stab},  or 
\item when  the domain $\Omega$ is  Lipschitz and the family of averages  $\{A_\eps\}_{\eps>0}$  satisfies  the  uniform boundedness \eqref{as:stab2} and the comparison principle
\eqref{as:comppr} properties, 
\end{itemize}
where $u$ is the unique  solution of the Dirichlet problem
\begin{equation} \label{dirichlet0}
\left\{
\begin{array}{cccl}
-\Delta_{p} u & = & 0& \text{ in } \Omega\\
u &=& g & \text{ on } \partial{\Omega},
\end{array}
\right.
\end{equation}
\end{theorem}

The proof  of this theorem for $C^2$ domains is based on an extension of the method in  \cite{BS91},  once we have the \textit{strong uniqueness theorem} for the $p$-Laplacian (Proposition 3.2 in \cite{dTMP20}). The key  observation  is that the convergence of approximations that satisfy the asymptotic mean value  property, the
uniform boundedness   property  \eqref{as:stab2}, and the comparison principle  \eqref{as:comppr} depends only on the strong uniqueness principle for the limit operator, which is the $p$-Laplacian in this case.  \par
To apply the Barles-Souganidis method in smooth domains of the Heisenberg group $\mathbb{H}$, we need to establish the \textit{strong uniqueness principle} for the homogeneous $p$-Laplacian in the Heisenberg group
$\Delta^N_{\mathbb{H},p}$. In \cite{MS21} we were able to establish it only when the domain $\Omega=B$,  the Euclidean unit ball in $\mathbb{R}^3$. Thus, we concluded the convergence of general mean approximations only in this case. 
\par
For more general domains in $\mathbb{H}$ an obstruction to the application of this method is the presence of characteristic points. 
In the current paper, we pursue a different approach that does not rely on the strong uniqueness principle.
 First, we use  good properties of the fundamental solution of the $p$-Laplacian to establish convergence in  smooth ring domains with $p$-harmonic boundary data. Second,   we set up a boundary iteration suggested by the proof of sufficiency for the Wiener condition of boundary regularity. These steps are independent of each other. We will establish them for the case of the natural $p$-means in the Heisenberg group. \par

The notion of natural $p$-means in general topological measure spaces was introduced in \cite{IMW17}.
Let $X$ be a compact topological measure space endowed with a positive Radon measure $\nu$.  Given a function $u\in L^p(X)$ and $1< p\le \infty$, there exists a unique real value $\mu_p^X(u)$ such that
\begin{equation}\label{naturalpmean0}
\|u-\mu_p^X(u)\|_p=\min_{\lambda\in \R}\| u-\lambda\|_p   \,.
\end{equation}
We will call $\mu_p^X(u)$ the natural $p$-mean of $u$ in $X$. Note that the above definition extends to the case $p=1$, provided that $u$ is assumed to be continuous on $X$.
Existence, uniqueness, and several useful properties of the natural $p$-means were studied in \cite{IMW17}, where the AMVP for the natural $p$-means is established in the Euclidean case. \par

While for general $p$ there is no explicit formula for $\mu_p^X(u)$,  for the cases $p=1,2$, and $p=\infty$  we have:
\begin{equation*}
\begin{array}{rcl}\medskip
\mu_1(u)& =& \text{med}(u), \\ \medskip
\mu_2(u)& =& \dashint_X u(y) d\nu, \text{  and} \\ 
\mu_{\infty}(u)& =& \frac12 (\essinf_{y\in X} u(y)+\esssup_{y\in X} u(y)).
\end{array}
\end{equation*}
Consider  the family  of natural $p$-mean operators  $\{\mu_p( u, \eps)\}_{0<\eps<1}$ defined on functions
$u\in {\mathcal B}(\Omega_{E})$ as follows. For  $x\in \overline{\Omega}\subset\mathbb{R}^n$ and $B_{\eps}(x)$ the Euclidean ball of radius $\eps$ centered at $x$, we set 
\begin{equation} 
\mu_p(u, \eps)(x)= \mu^{\overline{B_{\eps}(x)}}_{p}(u).
\end{equation}
Observe that for any $u\in L^p(\Omega_E)$
 the function
$x\mapsto\mu_p(u, \eps)(x)$ is continuous  in $\overline\Omega$. This property 
is not shared by the tug-of-war means of type \eqref{dpp1}.\par

We can now rephrase Theorem 3.2 in \cite{IMW17}:
\begin{theorem}\label{aysmptotic10} The family  of natural p-mean operators  $\{\mu_p(\cdot, \eps)\}_{0<\eps<1}$  satisfies the asymptotic mean value property relative to  the $p$-Laplacian; that is,  for every $\phi\in C^\infty(\Omega_E)$ such that $\nabla \phi \not=0$, we have
\begin{equation}\label{eq:MVP20}
\mu_p(\phi, \eps)(x)= \phi(x) +\frac{\eps^2}{2(n+p)} \Delta_p^H \phi(x)+ o(\eps^2) , \; \text{as} \; \eps \to 0 ,
\end{equation}
where the constant in $o(\eps^2)$ can be taken uniformly for all $x\in \overline{\Omega}$.
\end{theorem} 
The  uniform stability \eqref{as:stab} property, the  uniform boundedness \eqref{as:stab2} property,   and the comparison principle \eqref{as:comppr} of the natural $p$-means in the Euclidean space were established in \cite{MS21}.
Thus, Theorems \ref{thm1} and \ref{aysmptotic10} can be combined to get the following theorem.
\begin{theorem}[Convergence of natural $p$-mean approximations in the Euclidean case, \cite{MS21}]\label{thm:main20}
Assume $p\in(1,\infty]$, $\Omega\subset \R^n$ is a bounded Lipschitz domain and $g\in C(\partial\Omega)$. For $
0 < \eps \leq 1$ and $A_{\eps} =  \mu_{p}( \cdot, \eps) $, let $u_{\eps}$ be the solution to the DPP  \eqref{eq:expDPP1}, where $G$ is a continuous extension of $g$ to $\Gamma_1$. Then,  we 
  have  
 \[
 u_\eps\to u   \text{  uniformly  in   } \overline{\Omega}  \text{  as  } \eps\to0,  
\]
where $u$ is the unique  solution of the Dirichlet problem \eqref{dirichlet0}.
\end{theorem}
In a recent paper,  Chandra, Ishiwata, Magnanini, and Wadade \cite{CIMW21} have also proved the convergence of the natural $p$-means in the Euclidean case. Their approach and our approach differ in the treatment at  the boundary, but the main results are essentially the same in the case of $\mathbb{R}^n$.

Recall that for the first Heisenberg group,  $\mathbb{H}=(\mathbb{R}^3, *)$, the group operation is given as
$$(x_1,x_2,x_3)*(y_1,y_2,y_3)=\left(x_1+y_1, x_2+y_2, x_3+y_3+\frac{1}{2}(x_1y_2-x_2y_1)\right).$$
The vector fields
$$X_1=\partial_{x_1}-\frac{x_2}{2}\partial_{x_3},\quad X_2=\partial_{x_2}+\frac{x_1}{2}\partial_{x_3}\quad \text{and}\quad
T=\partial_{x_3}$$
form a basis of the associated Lie algebra. We denote the horizontal gradient of a smooth function $u$ by $\nabla_{\mathbb{H}}{u}=(X_1u)X_1+(X_2u)X_2$, 
 the horizontal Laplacian by
$$\Delta_{\mathbb{H}} u=X_1^2u+X_2^2u,$$
the horizontal $ p$-Laplacian by
$$\Delta_{\mathbb{H},p} u = X_1(|\nabla_{\mathbb{H}}{u}|^{p-2}X_1u)+X_2(|\nabla_{\mathbb{H}}{u}|^{p-2}X_2u),$$
and the  normalized horizontal $\infty$-Laplacian by
$$\Delta_{\mathbb{H},\infty} u =\langle  D^{2,*}_{\mathbb{H}} u \, \frac{\nabla_{\mathbb{H}}{u}}{|\nabla_{\mathbb{H}}{u}|}\, , \,  \frac{\nabla_{\mathbb{H}}{u}}{|\nabla_{\mathbb{H}}{u}|}\rangle,$$
where $D^{2,*}_{\mathbb{H}}u$ denotes the symmetrized Hessian $(D^{2,*}_{\mathbb{H}}u)_{ij}=(X_iX_j+X_jX_i)/2$.
For a smooth function $u$, with non-vanishing horizontal gradient, we define the normalized  $p$-Laplacian as
$$ \Delta^N_{\mathbb{H},p} u = |\nabla_{\mathbb{H}}{u}|^{ 2-p} \Delta_{\mathbb{H},p} u = ( p-2)\Delta_{\mathbb{H},\infty} u +\Delta_{\mathbb{H}} u \, .$$
The Kor\'{a}nyi  smooth gauge, given by
$$
|x|_{\mathbb{H}}=|(x_1,x_2,x_3)|_{\mathbb{H}}=\left((x_1^2+x_2^2)^2+16x_3^2\right)^{\frac{1}{4}},$$
induces the left-invariant metric
$$d_{\mathbb{H}}(x,y)=|x^{-1}*y|_{\mathbb{H}}.$$
The Kor\'{a}nyi ball with center $x_0$ and radius $r$ will be denoted by $$B^{\mathbb{H}}_r(x_0)=\{ x\in\mathbb{H}\, |\, d_{\mathbb{H}}(x_0,x)<r \}.$$
The Heisenberg group $\mathbb{H}$ is unimodular and its Haar measure is the Lebesgue measure in $\mathbb{R}^3$.
We also have a family of anisotropic dilations $(\rho_{\lambda})_{\lambda>0}$, that are group isomorphisms
\begin{equation}\label{dilations}
\rho_\lambda(x)=\rho_{\lambda}(x_1,x_2,x_3)=(\lambda x_1,\lambda x_2,\lambda^2 x_3).
\end{equation}
The  homogeneous dimension of $\mathbb{H}$ is $Q=4$.\par

The natural $p$-means in the Heisenberg group $\mathbb{H}$ were studied in   \cite{MS21}.
Suppose that $\Omega\subset\mathbb{H}$ is a bounded domain. Fix $0\leq \eps<1$ and consider the $\eps$-boundary strip given by
$$\Gamma_{\eps}^{\mathbb{H}} = \{x\in \mathbb{H}\setminus \Omega\,\colon\, \text{dist}_{\mathbb{H}}(x, \Omega) \leq \eps\}.$$
 Define $\Omega_{E}= \Omega\, \cup\,  \Gamma^{\mathbb{H}}_{1}$ and denote by $\mathcal{B}(\Omega_E)$ and  $\mathcal{B}(\overline{\Omega})$ the set of real-valued bounded Lebesgue measurable functions on $\Omega_E$ and $\overline{\Omega}$, respectively.
We define the natural $p$-mean average operator in the Heisenberg group,
$$\mu^{\mathbb{H}}_p(\,\cdot\,, \eps):\mathcal{B}(\Omega_E)\longrightarrow \mathcal{B}(\overline{\Omega}) ,$$ 
given by
$$\mu^{\mathbb{H}}_p(\phi, \eps)(y)= \mu^{\overline{B^{\mathbb{H}}_\eps(y)}}_{p}(\phi)\quad \text{for all}\,\,\phi\in \mathcal{B}(\Omega_E)\,\,\text{and}\,\, y\in\overline\Omega,$$
where $\mu^{\overline{B^{\mathbb{H}}_\eps(y)}}_{p}(\phi)$ is defined as in \eqref{naturalpmean0}, using as Radon measure $\nu$ the Lebesgue measure on $X=\overline{B^{\mathbb{H}}_\eps(y)}$.
\par

Let us now describe the DPP associated to the natural $p$-means in $\mathbb{H}$.
Let $G:\Gamma^{\mathbb{H}}_{\eps}\to \mathbb{R}$ be a continuous function and consider the solutions of the following boundary value problem
\begin{equation}\label{eq:expH^10}
\left\{
\begin{array}{cccl}
u_\eps(x) & = & \mu^{\mathbb{H}}_{p}(u_\eps, \eps) (x) &  \text{  if  }  x\in{\Omega},\\
u_\eps(x)&=& G(x) & \text{ if } {x\in\Gamma^{\mathbb{H}}_{\eps}}.
\end{array}
\right.
\end{equation}
The stability and monotonicity of the natural $p$-means in $\mathbb{H}$ was established in \cite{MS21}, as well as the following AMVP with respect to the $p$-Laplacian
  
 \begin{lemma}[\cite{MS21}] \label{amvph}
 For $1<p<\infty$ define the constant
	\begin{equation*}\label{c_p}  
	c_p= \frac {2}{(p+2)(p+4)}\Big(\frac{\Gamma(\frac{p}4+\frac32)}{   \Gamma(\frac{p}4+1)}\Big)^2,
	\end{equation*}	where $\Gamma$ is the Euler Gamma function.
	Let  $u$ be a smooth function in $\Omega_E$ with $\nabla_{\mathbb{H}}u\neq 0$ in $\Omega_E$. Then, we have the expansion 
	\begin{equation*}
	\mu^{\mathbb{H}}_p(u,\eps)(x_0)=u(x_0)+c_p\, \eps^2\Delta^N_{\mathbb{H},p} u(x_0)+o(\eps^2) \quad\text{as}\quad \eps\to 0
	\end{equation*}
	for all $x_0\in\overline{\Omega}$.   
\end{lemma}
When $p=\infty$, the lemma also holds with $c_\infty=\lim_{p\to\infty} c_p= 1/2$.
We remark that this lemma  has been obtained  independently and in the case of general Carnot groups in \cite{AKPW20}.  
\par
To show convergence, we first study the case
when the boundary data is itself the restriction of a $p$-harmonic function with non-vanishing gradient. For tug-of-war means 
$$T(u, \eps)(x)=\frac{\alpha}2 \left(\sup_{B_{\eps}(x)} u_\eps+   \inf_{B_{\eps}(x)} u_\eps \right)+\beta\, \dashint_{B_{\eps}(x)}u_\eps(y)\, dy \, ,$$
this was first done using probability in \cite{MPR12}   in $\mathbb{R}^n$.
The proofs of our results in this paper are analytic and do not rely on probabilistic techniques.\par\medskip  
To go from continuous to semi-discrete, we  build sub - and super-solutions of the DPP  \eqref{eq:expH^10} from solutions to the continuous problem by using the following perturbation lemma.
\begin{lemma}[Perturbations for natural $p$-means,  $p>2$]\label{perturbation0}
	Let $\Omega'$ be an open set containing $\Omega$, and
	let $U$ be a function such that $\Delta^N_{\mathbb{H},p} U=0$ and $\nabla_{\mathbb{H}} U\neq 0$ in 
	$\Omega'$. Then, there exist $\hat{\eps}>0$, $s\geq 4$ and $q_0\in\mathbb{H}$ such that, denoting
	 $v(x)=|q_0^{-1}*x|^s_{\mathbb{H}} \,$, we have 
	
$$U+\eps v \textrm{ is a subsolution of } \eqref{eq:expH^10}$$ and  $$U-\eps v  \textrm{ is a supersolution of }\eqref{eq:expH^10}$$ in $\Omega_\eps$ for all $0<\eps<\hat{\eps}$.
\end{lemma}
For tug-of-war means the Perturbation Lemma above  is valid for $1<p<\infty$, and it is due to Lewicka \cite{L18, L20} in $\mathbb{R}^n$ and to \cite{LMR} in  $\mathbb{H}$.
The key to proving this lemma in $\mathbb{H}$ is a strengthening of the  expansion in $\eps$ of $\mu^{\mathbb{H}}_p (u,\eps)(x_0)$ that we are able to prove for $p>2$.
\begin{proposition}\label{keylemma0}
	Let $p\ge 2$ and $u$ be a smooth function in $\Omega_E$ with $\nabla_{\mathbb{H}}u\neq 0$ in $\Omega_E$. Then, there exist $C>0$ and $\hat{\eps}>0$ such that
	\begin{equation*}
	|\mu^{\mathbb{H}}_p (u,\eps)(x_0)-u(x_0)-c_p\, \eps^2\Delta^N_{\mathbb{H},p} u(x_0)|\leq C\eps^3
	\end{equation*}
	for all $0<\eps<\hat{\eps}$ and $x_0\in\overline\Omega$. In particular, the constants $C$ and $\hat{\eps}$ are uniform in $x_0\in\overline\Omega$ and depend only on $p$ and the derivatives of $u$. 
\end{proposition}
Note that we have replaced $o(\eps^2)$ by $O(\eps^3)$ when $p>2$.
We remark that the new argument in the proof of this Proposition can also be used to give an alternative proof of the second order expansion in Lemma \ref{amvph} in the 
case $1<p<2$, different than those  in \cite{AKPW20} and \cite{IMW17}.
\par
From Proposition \ref{keylemma0} the convergence when the boundary data is itself the restriction of a $p$-harmonic function with non-vanishing gradient follows.
\begin{proposition}\label{Prop:convergence}
	Let $\Omega '$ be an open set containing $\Omega$. Let
	$U$ be a function such that $\Delta^N_{\mathbb{H},p} U=0$ and $\nabla_{\mathbb{H}} U\neq 0$ in $\Omega '$ and let $u_\eps$ be the solution of \eqref{eq:expH^10} with boundary datum $U$, for $\eps >0$ sufficiently small. Then
	\begin{equation*}
	u_\eps \longrightarrow U \quad\text{uniformly in}\,\, \overline{\Omega}
	\end{equation*}
\end{proposition}

To consider more general domains, we define the following boundary regularity condition.
\begin{definition}\label{Hcorkscrew0} We say that a domain $\Omega\subset\mathbb{H}$ satisfies the exterior $\mathbb{H}$-corkscrew condition if
	there exists $\bar{\delta}>0$ and $\mu\in(0,1)$ such that for every $\delta\in(0,\bar{\delta})$ and $y\in\partial\Omega$ there exists a ball $B^{\mathbb{H}}_{\mu\delta}(z)$ strictly contained in $B^{\mathbb{H}}_\delta(y)\setminus\Omega$.
\end{definition}

It is known that domains with $C^{1,1}$ boundary in the Euclidean sense satisfy the exterior $\mathbb{H}$-corkscrew condition (see \cite{CG}, Theorem 14 for domains in the Heisenberg group and \cite{MM05}, Theorem 1.3 for the more general case of domains in step 2 Carnot groups). This regularity is optimal in the sense that for every $\alpha\in[0,1)$ there exist domains with  $C^{1,\alpha}$ boundary in the Euclidean sense that do not satisfy the condition in Definition \ref{Hcorkscrew0}
(see Example 8.2 in \cite{LMR}). 

For domains satisfying the exterior $\mathbb{H}$-corkscrew condition we first prove the following boundary estimate.

\begin{theorem}\label{TheoremBoundary}
	Let $\Omega$ be an open, bounded subset of $\mathbb{H}$ satisfying the exterior $\mathbb{H}$-corkscrew condition in Definition \ref{Hcorkscrew0} and $G\in C(\Gamma^{\mathbb{H}}_1)$.
	For $2\leq p<\infty$ let $u_\eps$ be the solution of \eqref{eq:expH^10} in $\Omega$ with boundary value $G$ on $\Gamma^{\mathbb{H}}_\eps$, for $0<\eps<1$.
	Given $\eta>0$ there exist $\delta_0=\delta_0(\eta, \mu, p)$ and $\eps_0=\eps_0(\eta, \delta,\mu)$ such that
	\begin{equation*}
	|u_\eps(x)-G(y)|\leq \eta,
	\end{equation*}
	for all $y\in\partial\Omega$, $x\in B^\mathbb{H}_{\delta_0}(y)\cap\Omega$ and $\eps\leq\eps_0$.
\end{theorem}

Once we have Theorem \ref{TheoremBoundary}, we are ready to state our main result. 
 \begin{theorem}[For the range $2\le p <\infty$]\label{maintheorem}
	Let $\Omega$ be an  open bounded subset of $\mathbb{H}$ satisfying the exterior $\mathbb{H}$-corkscrew condition and let $\Gamma^{\mathbb{H}}_\eps$ be its outer $\eps$-boundary, for $0<\eps<1$. Let $g\in C(\partial\Omega)$, $G\in C(\Gamma^{\mathbb{H}}_1)$ be a continuous extension of $g$ and $u_\eps$ be the solution of the DPP  \eqref{eq:expH^10} with boundary datum $G$. Then $u_\eps$ converges to $u$ uniformly on $\overline{\Omega}$ as $\eps\to 0$, where $u$ is the  solution to the Dirichlet problem
	\begin{equation*}
	\begin{cases}
	-\Delta^N_{\mathbb{H},p} u = 0 &\quad \text{in}\quad \Omega\\
	u = g &\quad \text{in}\quad \partial\Omega.
	\end{cases}
	\end{equation*}
\end{theorem}
The plan of the paper is as follows. In Section \S 2 we present various properties of the natural $p$-means in 
$\mathbb{H}$. The proofs of the perturbation Lemma \ref{perturbation0}  and of  Proposition \ref{keylemma0} are in Section \S 3, while the proofs of Lemma \ref{perturbation0} and Proposition \ref{Prop:convergence} are in Section \S 4. The proof of the key boundary continuity estimate Theorem \ref{TheoremBoundary} is in Section \S 5, which is the most technical section of the paper, and the proof of the main result is in Section \S 6. \par
The major technical difference between the results in \cite{MS21}, this article,  and \cite{CIMW21} is that in the latter the DPP is modified at points close to $\partial\Omega$.  As a consequence the   solution of the DPP exists and is continuous provided that $\Omega$ satisfies a regularity condition implied by the exterior sphere property. Thus, we both prove convergence of slightly different approximations to $p$-harmonic functions.
In our case, we  get existence of possibly discontinuous solutions of the DPP for a general bounded domain $\Omega$, but need boundary regularity to prove convergence. \par
Finally, we note that the validity of Proposition \ref{keylemma0} for $1<p<2$ would imply that Theorem \ref{maintheorem} also  holds for $1<p<2$. However, our current proof of Proposition \ref{keylemma0} requires $p\ge 2$. 

\section{Natural $p$-Means in $\mathbb{H}$}\label{22}

We now collect several results on the natural $p$-means in $\mathbb{H}$ that we will need later. For all $u$, $v\in\mathcal{B}(\Omega_E)$, $0<\eps\leq1$ and $y\in\Omega$, the following properties hold:
\begin{itemize}
	\item Continuity in the $L^{p}$-norm (Theorem 2.4 in  \cite{IMW17}):
	$$\left|\|u-\mu^{\mathbb{H}}_p(u, \eps)(y)\|_{L^p(B_\eps^{\mathbb{H}}(y))}-\|v-\mu^{\mathbb{H}}_p(v, \eps)(y)\|_{L^p(B_\eps^{\mathbb{H}}(y))}\right|\le \|u-v\|_{L^p(B_\eps^{\mathbb{H}}(y))}.$$ 
	 In particular, the function $y\mapsto \mu^{\mathbb{H}}_p(u, \eps)(y)$ is continuous in $\overline\Omega$. 
	\item Monotonicity  (Theorem 2.5 in \cite{IMW17}): 
	$$\text{if}\,\, u\leq v\,\,\text{a.e. on}\,\,B_\eps^{\mathbb{H}}(y), \,\,\text{then}\,\,\mu^{\mathbb{H}}_p(u, \eps)(y)\leq \mu^{\mathbb{H}}_p(v, \eps)(y).$$ 
	\item Affine invariance (Proposition 2.7 in \cite{IMW17}): for $c,\alpha\in \R$ it holds $$\mu^{\mathbb{H}}_p(\alpha u+c, \eps)(y)=\alpha\mu^{\mathbb{H}}_p(u, \eps)(y)+c.$$ 
	\item Rescaling (Corollary 2.3 in \cite{IMW17}):  Let $x\in\Omega$. Defining $u_{\eps,x}(z)=u(x+\eps z)$ for $z\in B^\mathbb{H}_1(0)$, we have
	$$\mu_p(u,\eps)(x)=\mu_p( u_{\eps,x},1)(0).$$
	\item Integral characterization (Theorem 2.1 in \cite{IMW17}): For $1< p<\infty$, $\mu^{\mathbb{H}}_p(u, \eps)(y)$ is the unique solution of $$\int_{B_\eps^{\mathbb{H}}(y)} \left\lvert u(z)-\mu^{\mathbb{H}}_p(u, \eps)(y)\right\rvert^{p-2}\left(u(z)-\mu^{\mathbb{H}}_p(u, \eps)(y)\right)\dz =0,$$
	with the convention that $0^{p-2}0=0$ if $1< p<2$.
\end{itemize}
All the properties above extend to the case $p=1$ if we assume that $u$ and $v$ are continuous on $\Omega_E$.\par


Existence and uniqueness of the solution of the DPP \eqref{eq:expH^10} can be proved as in the Euclidean case. Solutions are automatically continuous in the interior (see Lemma 5.6 and Theorem 3.4 in \cite{MS21}):
\begin{lemma}
	  Let  $1<p\leq \infty$. There exists a unique $u_\eps\in\mathcal{B}(\Omega_\eps)$ that satisfies \eqref{eq:expH^10}. Moreover,  we have that  $u_\eps\in C(\Omega)$.
\end{lemma}

Note, however, that $u_\eps$ might not be continuous on $\overline{\Omega}$.

Next, we provide a version of the comparison principle which will be needed later. This is a slight modification of Theorem 4.3 in \cite{MS21}.

\begin{lemma}\label{UpgradedComparison}
	Let $u_\eps \in {\mathcal B}(\Omega_{\eps})$ be a subsolution of \eqref{eq:expH^10} in $\Omega$ with boundary datum $F$, and $v_\eps \in {\mathcal B}(\Omega_{\eps})$ be a supersolution  of \eqref{eq:expH^10} in $\Omega$ with boundary datum $G$. Then 
	\begin{equation*}
	u_\eps \leq v_\eps+\sup_{\Gamma^{\mathbb{H}}_\eps}(F-G)\quad\text{on}\quad \Omega_\eps.
	\end{equation*}
\end{lemma}

\begin{proof}
	Let $\phi_\eps=u_\eps-v_\eps$ be defined on $\Omega_\eps$. Note that $\phi_\eps\leq F-G$ on $\Gamma^{\mathbb{H}}_\eps$ and therefore $\displaystyle{\sup_{\Gamma^{\mathbb{H}}_\eps}\phi_\eps\leq\sup_{\Gamma^{\mathbb{H}}_\eps}(F-G)}$. Assume by contradiction that 
	$$M_\eps:=\sup_{\Omega}\phi_\eps>\sup_{\Gamma^{\mathbb{H}}_\eps}(F-G).$$
	Note that this implies $\sup_{\Omega_\eps}\phi_\eps=\sup_{\Omega}\phi_\eps=M_\eps$.
	By definition of $M_\eps$, there exists a sequence $x_n\in\Omega$ such that $\phi_\eps(x_n)$ converges to $M_\eps$. We can assume that, up to a subsequence, $x_n$ converges to some $x_0\in\overline{\Omega}$. Passing to the limit in the inequality
	\begin{equation*}
	\phi_\eps(x_n)\leq \mu_p^{\mathbb{H}}(u_\eps, \eps)(x_n)-\mu_p^{\mathbb{H}}(v_\eps, \eps)(x_n) \, ,
	\end{equation*}
	we obtain
	\begin{equation*}
	M_\eps\leq \mu_p^{\mathbb{H}}(u_\eps, \eps)(x_0)-\mu_p^{\mathbb{H}}(v_\eps, \eps)(x_0),
	\end{equation*}
	because the function $x\to\mu_p^{\mathbb{H}}(f,\eps)(x)$ is continuous in $\overline{\Omega}$ for $f\in L^p(\Omega_\eps)$. Therefore, 
	\begin{equation}\label{muUepsVeps}
	\begin{split}
	\mu_p^{\mathbb{H}}(u_\eps, \eps)(x_0)
	&\geq \mu_p^{\mathbb{H}}(v_\eps, \eps)(x_0)+M_\eps\\
	&=\mu_p^{\mathbb{H}}(v_\eps+M_\eps, \eps)(x_0).
	\end{split}
	\end{equation}
	Using the notation $h(u,\lambda)=|u-\lambda|^{p-2}(u-\lambda)$, we have that $h$ is increasing in $u$ for fixed $\lambda$, and decreasing in $\lambda$ for fixed $u$. Therefore, in $\Omega_\eps$, we have
	\begin{equation*}
	\begin{split}
	h\left(u_\eps,\mu_p^{\mathbb{H}}(u_\eps, \eps)(x_0)\right)
	&\leq h\left( u_\eps, \mu_p^{\mathbb{H}}(v_\eps+M_\eps, \eps)(x_0)\right)\\
	&\leq h\left(v_\eps+M_\eps, \mu_p^{\mathbb{H}}(v_\eps+M_\eps, \eps)(x_0)\right),
	\end{split}
	\end{equation*}
	because $u_\eps-v_\eps\leq M_\eps$ on $\Omega_\eps$. We can rewrite the previous inequality as
	\begin{equation*}
	h\left(u_\eps,\mu_p^{\mathbb{H}}(u_\eps, \eps)(x_0)\right)- h\left(v_\eps+M_\eps, \mu_p^{\mathbb{H}}(v_\eps+M_\eps, \eps)(x_0)\right) \leq 0
	\end{equation*}
	in $\Omega_\eps$. The integral characterization of the natural $p$-means implies that
	\begin{equation*}
	\int_{B^\mathbb{H}_\eps(x_0)} h\left(u_\eps,\mu_p^{\mathbb{H}}(u_\eps, \eps)(x_0)\right)- h\left(v_\eps+M_\eps, \mu_p^{\mathbb{H}}(v_\eps+M_\eps, \eps)(x_0)\right) =0,
	\end{equation*}
	and therefore 
	\begin{equation*}
	h\left(u_\eps,\mu_p^{\mathbb{H}}(u_\eps, \eps)(x_0)\right)- h\left(v_\eps+M_\eps, \mu_p^{\mathbb{H}}(v_\eps+M_\eps, \eps)(x_0)\right) =0 \quad\text{a.e. in}\,\,B^\mathbb{H}_\eps(x_0).
	\end{equation*}
	Since $s\to |s|^{p-2}s$ is injective, we obtain 
	\begin{equation*}
	u_\eps-\mu_p^{\mathbb{H}}(u_\eps, \eps)(x_0)= v_\eps+M_\eps- \mu_p^{\mathbb{H}}(v_\eps+M_\eps, \eps)(x_0) =0 \quad\text{a.e. in}\,\,B^\mathbb{H}_\eps(x_0),
	\end{equation*} 
	which rewrites as
	\begin{equation*}
	\begin{split}
	u_\eps-v_\eps
	&=M_\eps+\mu_p^{\mathbb{H}}(u_\eps, \eps)(x_0)- \mu_p^{\mathbb{H}}(v_\eps+M_\eps, \eps)(x_0).
	\end{split}
	\end{equation*} 
	Recalling \eqref{muUepsVeps}, we get
	$u_\eps-v_\eps\geq M_\eps =\sup_{\Omega_\eps}(u_\eps-v_\eps)$ a.e. in $B^\mathbb{H}_\eps(x_0)$, hence
	\begin{equation*}
	u_\eps-v_\eps=M_\eps\quad\text{a.e. in}\,\,B^\mathbb{H}_\eps(x_0).
	\end{equation*}
	The previous argument shows that if $x\in\bar{\Omega}$ is such that $\phi_\eps(x)=M_\eps$, then $\phi_\eps=M_\eps$ a.e. in $B^\mathbb{H}_\eps(x)$. Since $\Omega$ is bounded and connected, it is possible to find a finite chain of balls $\{B^\mathbb{H}_\eps(x_i)\}_{i=0}^N$ starting from $x_0$, such that $x_{i+1}$ is a point in $B^\mathbb{H}_\eps(x_i)$ with $\phi_\eps(x_{i+1})=M_\eps$ for all $i=0, ..., N-1$ and $B^\mathbb{H}_\eps(x_N)\cap\Gamma^{\mathbb{H}}_\eps\neq \emptyset$. This means that $\phi_\eps=M_\eps=\sup_{\Omega}\phi_\eps$ a.e. on $B^\mathbb{H}_\eps(x_N)\cap\Gamma^{\mathbb{H}}_\eps$, contradicting $\phi_\eps\leq F-G$ on $\Gamma^{\mathbb{H}}_\eps$.
\end{proof}

\section{Taylor Expansions}

The goal of this section is to prove  Proposition \ref{keylemma0}. 

First, we include details of the proof of Lemma \ref{amvph} in the
relevant case for our analysis $p>2$, because similar arguments will be used later to prove the stronger version Proposition \ref{keylemma0}.
We pay particular attention to the dependence of the estimate on  $x_0\in \overline{\Omega}$.  
\begin{proof}
	Let $p>2$ and denote $h(s)=|s|^{p-2}s$.
	Let $x_0\in\Omega$ and $0<\eps<1$. Define $u_{\eps, x_0}(z)=u(x_0*\rho_\eps(z))$ for all $z=(z_1, z_2, z_3)=(z_h,z_3)\in B$, where $B=B_1^{\mathbb{H}}(0)$ denotes the Kor\'{a}nyi ball centered at the origin with radius $1$. The Heisenberg Taylor expansion gives
	\begin{equation*}
	u_{\eps, x_0}(z)=u(x_0)+\eps\langle\nabla_{\mathbb{H}}u(x_0), z_h\rangle+\eps^2 Tu(x_0)z_3+\frac{1}{2}\eps^2\langle D^2_{\mathbb{H}} u(x_0)z_h, z_h \rangle+O(\eps^3),
	\end{equation*}
	where the constant in $O(\eps^3)$ is uniform in $x_0\in\Omega$ and depends on the $L^\infty(\Omega)$ norm of the second and third derivatives of $u$.
	Therefore, by the affine invariance and monotonicity of the natural means, it is enough to show the expansion for the quadratic function
	
	\begin{equation*}
	q_{\eps, x_0}(z)=u(x_0)+\eps\langle \nabla_{\mathbb{H}}u(x_0), z_h\rangle+\eps^2 Tu(x_0)z_3+\frac{1}{2}\eps^2\langle D^2_{\mathbb{H}} u(x_0)z_h, z_h \rangle.
	\end{equation*}
	
	\medskip
	
	{\bf 1.} First order term.
	Let 
	\begin{equation*}
	v_{\eps, x_0}(z)
	=\frac{q_{\eps, x_0}(z)-q_{\eps, x_0}(0)}{\eps}
	=\langle \nabla_{\mathbb{H}} u(x_0), z_h \rangle+\eps Tu(x_0)z_3+\frac{1}{2}\eps \langle D^2_{\mathbb{H}} u(x_0)z_h, z_h \rangle \, ,
	\end{equation*}
	and note that $v_{\eps, x_0}$ converges uniformly on $B$ to the function $v_{x_0}(z)=\langle \nabla_{\mathbb{H}} u(x_0), z_h \rangle$ as $\eps\to 0$. As a consequence,
	\begin{equation*}
	\frac{\mup(q_{\eps, x_0}, 1)(0)-q_{\eps, x_0}(0)}{\eps}=\mup(v_{\eps, x_0},1)(0)\longrightarrow \mup(v_{x_0},1)(0)=0,
	\end{equation*}
	where the last equality is due to 
	\begin{equation*}
	\int_B h(v_{x_0}(z))\dz=\int_B |\langle \nabla_{\mathbb{H}} u(x_0), z_h\rangle|^{p-2}\langle\nabla_{\mathbb{H}} u(x_0), z_h\rangle\dz=0,
	\end{equation*}
	which holds by symmetry of the integrand on $B$.
	
	\medskip
	
	{\bf 2.} Second order term.
	Let
	\begin{equation*}
	\delta_\eps(x_0)=\frac{\mup(q_{\eps, x_0},1)(0)-q_{\eps, x_0}(0)}{\eps^2}=\frac{\mup(v_{\eps,x_0},1)(0)}{\eps
	}.
	\end{equation*}
	By the integral characterization of the natural means we have
	\begin{equation}\label{firstOrderExpansion}
	\begin{split}
	0
	&= \int_B h(v_{\eps,x_0}(z)-\mup(v_{\eps,x_0},1))\dz\\
	&=\int_B h\left(\langle \nabla_{\mathbb{H}}u(x_0), z_h\rangle+\eps Tu(x_0)z_3+\frac{1}{2}\eps\langle D^2_{\mathbb{H}} u(x_0)z_h, z_h\rangle-\eps\delta_\eps(x_0)\right)\dz\\
	&= \int_B h(v_{x_0}(z))\dz +\eps\int_B \left(\int_0^1h'(F(z,x_0,\eps,t))\dt\right)  \psi(z,x_0,\eps)\dz,
	\end{split}
	\end{equation}
	where we used a first order Taylor expansion and denoted 
	\begin{equation}\label{psi}
	\psi(z,x_0,\eps)=Tu(x_0)z_3+\frac{1}{2}\langle D^2_{\mathbb{H}} u(x_0)z_h, z_h\rangle-\delta_\eps(x_0)
	\end{equation}
	and
	\begin{equation*}
	F(z,x_0,\eps,t)=\langle\nabla_{\mathbb{H}}u(x_0), z_h\rangle+t\eps\psi(z,x_0,\eps).
	\end{equation*}
	Manipulating \eqref{firstOrderExpansion}, we can explicitly compute 
	\begin{equation*}
	\delta_\eps(x_0)=\frac{\displaystyle{\int_B \int_0^1h'(F(z,x_0,\eps,t))\dt\left(Tu(x_0)z_3+\frac{1}{2}\langle D^2_{\mathbb{H}} u(x_0)z_h, z_h\rangle\right)\dz}}{\displaystyle{\int_B\int_0^1h'(F(z,x_0,\eps,t))\dt\dz}},
	\end{equation*} 
	which implies the uniform bound
	\begin{equation*}
	|\delta_\eps(x_0)|\leq \norm{Tu}_{L^{\infty}(\Omega)}+\norm{D^2_{\mathbb{H}}u}_{L^{\infty}(\Omega)}
	\end{equation*} for all $0<\eps<1$ and $x_0\in\overline{\Omega}$. Since $p>2$, by the dominated convergence theorem, up to a subsequence,  we obtain that
	\begin{equation*}
	\delta_\eps(x_0)\longrightarrow \frac{\displaystyle{\int_B\left( Tu(x_0)z_3+\frac{1}{2}\langle D^2_{\mathbb{H}} u(x_0)z_h, z_h\rangle \right) h'(\langle \nabla_{\mathbb{H}} u(x_0), z_h\rangle)\dz}}{\displaystyle{\int_B h'(\langle\nabla_{\mathbb{H}} u(x_0), z_h\rangle)\dz}}=:\delta_0(x_0).
	\end{equation*}
	 This integral can be explicitly computed (\cite{IMW17},   \cite{AKPW20}) to obtain
		\begin{equation*}
		\delta_0(x_0)=c_p\, \Delta^N_{\mathbb{H},p} u(x_0).
		\end{equation*}

\end{proof}


We show now the proof of Proposition \ref{keylemma0}.

\begin{proof}
	
	Using the notation in  the proof of Lemma \ref{amvph}, it is enough to
	show that the quantity
	$$\frac{\delta_\eps(x_0)-\delta_0(x_0)}{\eps}$$
	is uniformly bounded in $0<\eps<1$ and $x_0\in \overline{\Omega}$ .
	
	To this end, in \eqref{firstOrderExpansion} use a second order Taylor expansion to get
	\begin{equation*}
	\begin{split}
	0
	&= \int_B h'(\langle\nabla_{\mathbb{H}} u(x_0), z_h\rangle)\psi(z,x_0,\eps)\dz \\
	&+\eps\int_B  \left(\int_0^1h''(F(z,x_0,\eps,t))\dt \right)\psi^2(z,x_0,\eps) \dz.
	\end{split}
	\end{equation*}
	Therefore,
	\begin{equation*}
	\frac{\delta_\eps(x_0)-\delta_0(x_0)}{\eps} 
	= \frac{\displaystyle{\int_B \left(\int_0^1h''(F(z,x_0,\eps,t))\dt \right)\psi^2(z,x_0,\eps)\dz}}{\displaystyle{\int_B h'(\langle\nabla_{\mathbb{H}} u(x_0), z_h\rangle)\dz}}.
	\end{equation*}
	
	\medskip 
	
	First, we establish a lower bound on the denominator.
	Let $R_{x_0}$ be a rotation in the horizontal plane such that $R_{x_0}^{\tau}\nabla_{\mathbb{H}} u(x_0)=|\nabla_{\mathbb{H}} u(x_0)|(1,0) =(|\nabla_{\mathbb{H}} u(x_0)|,0)$. Observe that $B$ is invariant under the change of variables $z=R_{x_0}y$, so 
	\begin{equation*}
	\begin{split}
	\int_B h'(\langle\nabla_{\mathbb{H}} u(x_0), z_h\rangle)\dz
	&=(p-1)\int_B|\langle\nabla_{\mathbb{H}} u(x_0), z_h\rangle|^{p-2}\dz\\
	&= (p-1)|\nabla_{\mathbb{H}} u(x_0)|^{p-2}\int_B|y_1|^{p-2}\dy \\
	&\geq c_p \min_{\overline{\Omega}}|\nabla_{\mathbb{H}} u|^{p-2},
	\end{split}
	\end{equation*}
	because $0<\int_B|y_1|^{p-2}\dy <\infty$ for $p>2$.
	
	\medskip 
	
	Now we establish an upper bound on the numerator. Note that $$\psi^2(z,x_0,\eps)\leq 4\left(\norm{Tu}_{L^{\infty}(\Omega)}+\norm{D^2_{\mathbb{H}}u}_{L^{\infty}(\Omega)}\right)^2$$ for all $0<\eps<1$, $z\in B$, and $x_0\in \overline{\Omega}$. Since $h''(s)=(p-1)(p-2)|s|^{p-4}s$, it is enough to estimate the integral
	\begin{equation*}
	\begin{split}
	\int_B \int_0^1 \left\lvert \langle\nabla_{\mathbb{H}}u(x_0), z_h\rangle+t\eps\psi(z,x_0,\eps) \right\rvert^{p-3} \dt\dz.
	\end{split}
	\end{equation*}
	After performing the change of variables $z=R_{x_0}y$, where $R_{x_0}$ is the same rotation as above, the integral becomes
		\begin{equation*}
	\begin{split}
	&\int_B \int_0^1 \Big\lvert y_1|\nabla_{\mathbb{H}}u(x_0)|+t\eps\psi(R_{x_0}y,x_0,\eps) \Big\rvert^{p-3} \dt\dy\\
	&=|\nabla_{\mathbb{H}}u(x_0)|^{p-3}\int_B \int_0^1 | y_1+t\eps \Psi(y,x_0,\eps) |^{p-3} \dt\dy,
	\end{split}
	\end{equation*}
	where we denoted 
	\begin{equation*}
	\begin{split}
	\Psi(y,x_0,\eps)&=|\nabla_{\mathbb{H}}u(x_0)|^{-1}\psi(R_{x_0}y,x_0,\eps)\\
	&=|\nabla_{\mathbb{H}}u(x_0)|^{-1}\left(Tu(x_0)y_3+\frac{1}{2}\langle R_{x_0}^{\tau} D^2_{\mathbb{H}} u(x_0)R_{x_0}y_h, y_h\rangle-\delta_\eps(x_0)\right) .
	\end{split}
	\end{equation*}
	Observe that 
	\begin{equation*}
	|\Psi(y,x_0,\eps)|\leq \frac{2\norm{Tu}_{L^{\infty}(\Omega)}+\norm{D^2_{\mathbb{H}}u}_{L^{\infty}(\Omega)}}{{\min_{\overline{\Omega}}|\nabla_{\mathbb{H}} u|}}=c_\Psi
	\end{equation*}
	for all $y\in B$, $0<\eps<1$ and $x_0\in \overline{\Omega}$.
	
	For fixed $\eps$, $t\in[0,1]$ and $x_0\in \overline{\Omega}$, consider the change of variables $\zeta=\Phi_{\eps,t,x_0}(y)=\Phi_\eps(y)$, given by
	\begin{equation*}
	\begin{cases}
	\zeta_1=y_1+\eps t \Psi(y,x_0,\eps)\\
	\zeta_i=y_i \quad i=2,3
	\end{cases}
	\end{equation*}
	for $y\in B$.
	
	\medskip
	We claim that there exists $\hat{\eps}>0$ such that for all $0<\eps<\hat{\eps}$ the map $\Phi_\eps: B\longrightarrow \Phi_\eps(B)$ is a diffeomorphism with $\min_B \mathcal{J}\Phi_\eps\geq \frac{1}{2}$.
	More exactly, $\hat{\eps}$ can be taken to be 
	\begin{equation*}
	\hat{\eps}=\frac{1}{2}\left( \frac{\norm{Tu}_{L^{\infty}(\Omega)}+\norm{D^2_{\mathbb{H}}u}_{L^{\infty}(\Omega)}}{\min_{\overline{\Omega}}|\nabla_{\mathbb{H}} u|}+1 \right)^{-1} \, ,
	\end{equation*}
	so it is independent of $x_0\in \overline{\Omega}$.
	Indeed, the Euclidean Jacobian of the transformation is 
	\begin{equation*}
	\mathcal{J}\Phi_\eps(y) =I_{3\times 3}+\eps t \begin{pmatrix}
	\nabla_y \Psi(y, x_0, \eps)\\
	0\\
	0
	\end{pmatrix},
	\end{equation*}
	which implies 
	\begin{equation*}
	|\mathcal{J}\Phi_\eps(y)|=|1+\eps t \partial_{y_1} \Psi(y, x_0, \eps)|\geq \frac{1}{2}
	\end{equation*}
	for all $y\in B$, $0<\eps<\hat{\eps}$ and $x_0\in \overline{\Omega}$, because 
	\begin{equation*}
	\norm{\partial_{y_1} \Psi(\cdot, x_0, \eps)}_{L^\infty(\Omega)}\leq \frac{\norm{Tu}_{L^{\infty}(\Omega)}+\norm{D^2_{\mathbb{H}}u}_{L^{\infty}(\Omega)}}{\min_{\overline{\Omega}}|\nabla_{\mathbb{H}} u|}.
	\end{equation*}
	Moreover, for all $0<\eps<\hat{\eps}$ and $\eta$, $\xi\in B$ we have
	\begin{equation*}
	\begin{split}
	|\Phi_\eps(\xi)-\Phi_\eps(\eta)|\geq |\xi-\eta|-\eps t |\Psi(\xi, x_0, \eps)-\Psi(\eta, x_0, \eps)|
	\geq \frac{1}{2}|\xi-\eta|,
	\end{split}
	\end{equation*}
	because $ |\Psi(\xi, x_0, \eps)-\Psi(\eta, x_0, \eps)|\leq \norm{\nabla_{y} \Psi(\cdot, x_0, \eps)}_{L^\infty(\Omega)}|\xi-\eta| $. This concludes the proof of the claim.
	
	\medskip
	
	Now, for $y\in B$, $\zeta\in\Phi_\eps(B)$ we have $|\zeta_2|=|y_2|\leq 1$, $|\zeta_3|=|y_3|\leq 1$ and $|\zeta_1|=|y_1+\eps t \Psi(y, x_0, \eps)|\leq |y_1|+c_\Psi\eps \leq 1+c_\Psi\eps $, therefore
	\begin{equation*}
	\begin{split}
	\int_B |y_1+\eps t \Psi(y, x_0, \eps)|^{p-3}\dy
	&= \int_{\Phi_\eps(B)}|\zeta_1|^{p-3} |\mathcal{J}\Phi_\eps(\Phi_\eps^{-1}(\zeta))|^{-1}\,\mbox{d}\zeta\\
	&\leq 2 \int_{\Phi_\eps(B)}|\zeta_1|^{p-3} ,\mbox{d}\zeta\\
	&\leq 2 \int_{\{|\zeta_1|\leq 1+c_\Psi\eps \}} |\zeta_1|^{p-3} \,\mbox{d}\zeta_1\\
	&=4(1+c_B\eps )^{p-2},
	\end{split}
	\end{equation*}
	since $p>2$.
\end{proof}

\section{Convergence in the case of $p$-harmonic data}

In this section we  prove Lemma \ref{perturbation0} and Proposition \ref{Prop:convergence}. 

In Lemma \ref{perturbation0} we show that we can obtain supersolutions and subsolutions of the DPP \eqref{eq:expH^10} by small perturbations of $p$-harmonic functions with non-vanishing horizontal gradient. This will be used as a first step towards proving convergence of the approximation scheme to the appropriate solution of the Dirichlet problem for the $p$-Laplace equation, in the case its boundary datum is $p$-harmonic.

First we show the proof of Lemma \ref{perturbation0}.

\begin{proof}
	Up to a left translation, we can assume that $\Omega$ does not intersect the cylinder $$\{(x_1,x_2,x_3)\in\mathbb{H}\,\,|\,\, x_1^2+x_2^2<1\},$$ so we can choose $q_0=0$ and $v(x)=|{x}|_{\mathbb{H}}^s$.
	{ From the calculations in the proof of Theorem 12.1 in \cite{LMR}} there exist $s\geq 4$ and $\hat{\eps}$ such that 
	\begin{equation*}
	\Delta^N_{\mathbb{H},p} (U+\eps v)\geq s\eps \quad \text{in}\,\,\Omega
	\end{equation*}
	for all $0<\eps<\hat{\eps}$.
	From the expansion in Proposition \ref{keylemma0}, there exists $C>0$ such that
	\begin{equation*}
	\begin{split}
	(U+\eps v)(x) 
	&\leq \mu_p^{\mathbb{H}}(U+\eps v, \eps)(x)-c_p\eps^2\Delta^N_{\mathbb{H},p}(U+\eps v)(x)+C\eps^3\\
	&\leq \mu_p^{\mathbb{H}}(U+\eps v, \eps)(x)-\eps^3(c_ps-C)\\
	&\leq \mu_p^{\mathbb{H}}(U+\eps v, \eps)(x)
	\end{split}
	\end{equation*}
	for all $x\in\Omega$ and $0<\eps<\hat{\eps}$, provided we choose $s>C/c_p$ and further restrict $\hat{\eps}$. 
	
	Analogous computations give
	$\Delta^N_{\mathbb{H},p} (U-\eps v)\leq -s\eps$ in $\Omega$, which, again by Proposition \ref{keylemma0}, implies $(U+\eps v)(x) \leq \mu_p^{\mathbb{H}}(U+\eps v, \eps)(x)$ for all $x\in\Omega$ and an appropriate choice of $s$.
\end{proof}

Next comes the proof of Proposition \ref{Prop:convergence}.
\begin{proof}
	By Lemma \ref{perturbation0}, there exist $\hat{\eps}>0$, $s\geq 4$ and $q_0\in\mathbb{H}$ such that, denoting $v(x)=|q_0^{-1}*x|_{\mathbb{H}}^s$ we have that $U+\eps v$ is a subsolution and $U-\eps v$ is a supersolution of \eqref{eq:expH^10} in $\Omega_\eps$ for all $0<\eps<\hat{\eps}$.
	By the comparison principle in Lemma \ref{UpgradedComparison} we get
	\begin{equation*}
	U+\eps v \leq u_\eps +\eps\sup_{\Gamma^{\mathbb{H}}_\eps}(v)
	\end{equation*}
	and
	\begin{equation*}
	u_\eps\leq U-\eps v +\eps\sup_{\Gamma^{\mathbb{H}}_\eps}(v)
	\end{equation*}
	on $\Omega_\eps$, for all $0<\eps<\hat{\eps}$	. Therefore, 
	\begin{equation*}
	|u_\eps-U|\leq -\eps v+\eps\sup_{\Gamma^{\mathbb{H}}_\eps}(v) \leq 2\eps \sup_{\Omega_1}v
	\end{equation*}
	on $\Omega$, so 
	\begin{equation*}
	\norm{u_\eps-U}_{L^\infty(\Omega)}\leq C_\Omega \eps,
	\end{equation*}
	where $C_\Omega=2 \sup_{\Omega_1}v$. This uniform bound concludes the proof.
\end{proof}

\section{Boundary Estimate}

In this section we prove Theorem \ref{TheoremBoundary}.

\begin{proof} Fix $\eta>0$.	We prove the upper bound $u_{\eps} (x) \leq G(y) + \eta$, the proof of the lower one being analogous.
	 By uniform continuity, there exists $\delta\in(0,\bar{\delta})$ such that
	\begin{equation}\label{GxGy}
	G(x)\leq G(y)+\frac{\eta}{2} \quad\text{for all}\,\, y\in\partial\Omega, \, x\in B^\mathbb{H}_{5\delta}(y)\cap \Gamma^{\mathbb{H}}_\eps \,\,\text{and}\,\, 0<\eps<1.
	\end{equation}
	
	Introducing the notations 
	\begin{equation}\label{Beps}
	N^\eps(y)=N^{\eps, \delta}(y) :=\sup_{B^\mathbb{H}_{5\delta}(y)\cap \Gamma^{\mathbb{H}}_\eps}G  \quad
	\text{and}\quad 
	M^\eps :=\sup_{ \Gamma^{\mathbb{H}}_\eps}G,
	\end{equation}
	inequality \eqref{GxGy} rewrites as
	\begin{equation}\label{BF}
	N^\eps(y)\leq G(y)+\frac{\eta}{2}\quad\text{for all}\,\, y\in \partial\Omega\, \,\text{and}\,\,0<\eps<1.
	\end{equation}
	
	Set $\xi=\frac{4- p}{ p-1}$ and
	\begin{equation}\label{theta}
	\theta = \frac{1-\frac{1}{2}\left(\frac{\mu}{2-\mu}\right)^{\xi}-\frac{1}{2}\left(\frac{\mu}{2}\right)^\xi}{1-\left(\frac{\mu}{2}\right)^\xi}\in(0,1).
	\end{equation}
	
	Note that $\theta$ depends only on $\mu$ and $ p$.
	For $k\geq 0$ define 
	\begin{equation}\label{deltak}
	\delta_{k}=\frac{\delta}{4^{k-1}}
	\end{equation}  
	and
	\begin{equation}\label{Meps}
	M_k^\eps(y)=N^\eps(y)+\theta^k(M^\eps-N^\eps(y)).
	\end{equation}
	
	We have the following
	
	\medskip
	{\bf Claim} 
	Let $\eps_k>0$ and suppose that for all $\eps\in(0,\eps_k)$ we have
	\begin{equation*}
	u_\eps \leq M_k^\eps(y) \quad \text{in}\quad B^\mathbb{H}_{\delta_k}(y)\cap \Omega.
	\end{equation*}
	Then, there exists $\eps_{k+1}=\eps_{k+1}(\eta, \mu,\delta, p)\in(0,\eps_k)$  such that
	\begin{equation*}
	u_\eps \leq M_{k+1}^\eps(y) \quad \text{in}\quad B^\mathbb{H}_{\delta_{k+1}}(y)\cap \Omega
	\end{equation*}
	for all $\eps\in(0, \eps_{k+1})$. 
	
	\medskip
	
	Using the claim above, since $u_\eps\leq M^\eps=M_0^\eps$ in $\Omega_\eps$ for all $\eps\in(0,1)$, we can find $\eps_1>0$ such that $u_\eps\leq M_{1}^\eps(y)$ in $B^\mathbb{H}_{\delta_1}(y)\cap\Omega$ for $\eps\leq\eps_1$. We now repeat this process, and after $k_0$ iterations we find $\eps_{k_0}>0$ such that  $u_\eps\leq M_{k_0}^\eps(y)$ in $B^\mathbb{H}_{\delta_{k_0}}(y)\cap\Omega$ for $\eps\in(0,\eps_{k_0})$. 
	Choosing $k_0\in\mathbb{N}$ such that
	$$k_0>\log_{\theta}\left(\frac{\eta}{2}\left(\sup_{\Gamma^{\mathbb{H}}_1} G-\inf_{\Gamma^{\mathbb{H}}_1} G+1\right)^{-1}\right)$$
	we have
	\begin{equation}\label{k0}
	M_{k_0}^\eps(y)-N^\eps(y)=\theta^{k_0}(M^\eps-N^\eps(y))\leq \frac{\eta}{2},
	\end{equation}
	for all $y\in\partial\Omega$, because $\displaystyle{M^\eps\leq \sup_{\Gamma^{\mathbb{H}}_\eps} G}$ and $\displaystyle{N^\eps(y)\geq \inf_{\Gamma^{\mathbb{H}}_\eps} G}$.
	Combining with \eqref{BF} we get
	$$u_\eps\leq G(y)+\eta\quad\text{in}\quad B^\mathbb{H}_{\delta_{k_0}}(y)\cap\Omega,$$ for all $\eps\in(0,\eps_{k_0})$.
	
	
	To conclude, we now provide a sketch of the proof of the claim.
	The key properties of having radial fundamental solutions in annular domains for $\Delta^N_{\mathbb{H},p}$ with non-vanishing horizontal gradient, 
	together with the comparison principle for solutions of \eqref{eq:expH^10}, the uniform convergence in Proposition \ref{Prop:convergence} and the fact that a genuine triangle inequality holds for the Kor\'{a}nyi norm, allow us to adapt the proof of Lemma 4.4 of \cite{dTMP20} to the current setting with minor modifications.

	{Proof of claim:}\\
	{\bf  1.}
	Since $\Omega$ satisfies the exterior $\mathbb{H}$-corkscrew condition, there exists a sequence of balls $B^\mathbb{H}_{\mu \delta_{k+1}}(z_k)$ strictly contained in $B^\mathbb{H}_{\delta_{k+1}}(y)\setminus \Omega$ for all $k\in\mathbb{N}$. To simplify the notation, let $N=N^\eps(y)$ and $M_k=M^\eps_k(y)$.
	The following boundary value problem
	\begin{equation*}
	\begin{cases}
	-\Delta^N_{\mathbb{H},p} U_k = 0 &\quad \text{in}\quad B^\mathbb{H}_{\delta_k}(z_k)\setminus \overline{B^\mathbb{H}_{\mu\delta_{k+1}}(z_k)} \\
	U_k = N &\quad \text{on}\quad \partial B^\mathbb{H}_{\mu\delta_{k+1}}(z_k)\\
	U_k = M_k &\quad \text{on}\quad \partial B^\mathbb{H}_{\delta_{k}}(z_k)
	\end{cases}
	\end{equation*}
	has the radial solution
	\begin{equation*}
	U_k(x)=
	\begin{dcases}
	\frac{a_k}{|z_k^{-1}*x|_{\mathbb{H}}^{\xi}}+b_k &\quad\text{if}\,\, p\neq 4\\
	a_k\log(|z_k^{-1}*x|_{\mathbb{H}})+b_k &\quad\text{if}\,\, p= 4
	\end{dcases},
	\end{equation*}
	for suitable coefficients $a_k$ and $b_k$.  Denote by $\tilde{\Omega}$ the annulus  $B^\mathbb{H}_{\delta_k}(z_k)\setminus \overline{B^\mathbb{H}_{\mu\delta_{k+1}}(z_k)}$. Choosing $\eps_{k+1}$ sufficiently small, for $0<\eps<\eps_{k+1}$ we let $U_k^\eps$  be the solution of the problem \eqref{eq:expH^10} in $\tilde{\Omega}$ with boundary value $U_k$. Since $\nabla_{\mathbb{H}}U_k\neq 0$ in $\tilde{\Omega}_\eps$, Proposition \ref{Prop:convergence} implies that $U_k^\eps$ converges uniformly to $U_k$ in $\tilde{\Omega}$.
	
	\medskip
	
	{\bf 2.}
	Consider now $B^\mathbb{H}_{\delta_k/2}(z_k)$. Due to its radial nature, we have that $U_k\geq \alpha M_k+\beta N$ on $B^\mathbb{H}_{\delta_k/2}(z_k)$, for appropriate coefficients satisfying $\alpha+\beta=1$. By hypothesis, $\alpha u_\eps+\beta N\leq \alpha M_k+\beta N$ on the $\eps$-boundary of $B^\mathbb{H}_{\delta_k/2}(z_k)\cap\Omega$. Using the fact that  $U_k^\eps$ is uniformly close to $U_k$ and the comparison principle for solutions of \eqref{eq:expH^10}, given $\Gamma^{\mathbb{H}}>0$ we get that 
	\begin{equation}\label{alphabetaU}
	\alpha u_\eps+\beta N\leq U_k+2\Gamma^{\mathbb{H}} \quad\text{on}\,\,B^\mathbb{H}_{\delta_k/2}(z_k)\cap\Omega,
	\end{equation}
	for $\eps\in(0,\eps_{k+1})$ sufficiently small.
	\medskip
	
	{\bf 3.}
	Now consider $B^\mathbb{H}_{\delta_{k+1}}(y)$. Because a genuine triangle inequality holds for the Kor\'{a}nyi gauge, we have $B^\mathbb{H}_{\delta_{k+1}}(y) \subset B^\mathbb{H}_{\delta_{k}/2}(z_k)$ (so \eqref{alphabetaU} holds in this ball) and $B^\mathbb{H}_{\delta_{k+1}}(y) \subset B^\mathbb{H}_{(2-\mu)\delta_{k+1}}(z_k)$. We have that $U_k\leq\beta 'N+\alpha ' M_k$ on $B^\mathbb{H}_{(2-\mu)\delta_{k+1}}(z_k)$ for suitable coefficients satisfying $\alpha'+\beta'=1$. Combining with \eqref{alphabetaU}, in $B^\mathbb{H}_{\delta_{k+1}}(y)\cap\Omega$ and for $\eps\in(0,\eps_{k+1})$ we have
	\begin{equation*}
	\begin{split}
	u_\eps&\leq 
	\frac{\beta ' -\beta}{\alpha}N+\frac{\alpha '}{\alpha}M_k+\frac{2\Gamma^{\mathbb{H}}}{\alpha}\\
	&=N+\theta(M_k-N)=M_{k+1},
	\end{split}
	\end{equation*}
	with an appropriate choice of $\Gamma^{\mathbb{H}}$. Keeping track of the coefficients, whose expressions can be explicitly computed, we find that $\theta$ turns out to be as in \eqref{theta}.
\end{proof}


\section{The proof of Theorem \ref{maintheorem}}
Recall the notion of viscosity solutions of the $p$-Laplace equation
\begin{equation}\label{NormalizedPLaplacianEquation}
-\Delta^N_{\mathbb{H},p} v = 0 \quad\text{in}\quad\Omega.
\end{equation}

\begin{definition}
	Let $\Omega\subset\mathbb{H}$ be open and bounded. A bounded upper semicontinuous function $v:\Omega\longrightarrow\mathbb{R}$ is a viscosity subsolution of \eqref{NormalizedPLaplacianEquation} if for every $x_0\in\Omega$ and for every $\phi\in C^2(\Omega)$ such that $v-\phi$ has a strict maximum at $x_0$, $v(x_0)=\phi(x_0)$ and $\nabla_{\mathbb{H}}{\phi}(x_0)\neq 0$, it holds $$-\Delta^N_{\mathbb{H},p} \phi(x_0)\leq 0.$$

	A bounded lower semicontinuous function $v:\Omega\longrightarrow\mathbb{R}$ is a viscosity supersolution of \eqref{NormalizedPLaplacianEquation} if
 for every $x_0\in\Omega$ and for every $\phi\in C^2(\Omega)$ such that $v-\phi$ has a strict minimum at $x_0$, $v(x_0)=\phi(p_0)$ and $\nabla_{\mathbb{H}}{\phi}(x_0)\neq 0$, it holds $$-\Delta^N_{\mathbb{H},p} \phi(x_0)\geq 0.$$

	Finally, a viscosity solution of \eqref{NormalizedPLaplacianEquation} is a function which is both a  viscosity subsolution and a viscosity supersolution of \eqref{NormalizedPLaplacianEquation}. 
\end{definition}

We now study the convergence of $u_\eps$ as $\eps$ goes to zero, where $u_\eps$ is the solution of \eqref{eq:expH^10} with boundary value $G\in C(\Gamma^{\mathbb{H}}_1)$. For $x\in\Omega$ define 
\begin{equation}\label{barU}
\begin{split}
\overline{u}(x):&=\limsup_{\substack{\eps\to 0^+\\ q\to x }} u_\eps(q),\\
\underline{u}(x):&=\liminf_{\substack{\eps\to 0^+\\ q\to x }} u^\eps(q)
\end{split}
\end{equation}
Note that $\overline{u}$ is upper semicontinuous and $\underline{u}$ is lower semicontinuous.

\begin{theorem}
	Let $\overline{u}$ and $\underline{u}$ be as in \eqref{barU}. Then $\overline{u}$ is a viscosity subsolution and $\underline{u}$ is a viscosity supersolution of $-\Delta^N_{\mathbb{H},p} u = 0$ in $\Omega$.
\end{theorem}

\begin{proof}
	We give a proof for $\overline{u}$, the one for $\underline{u}$ being analogous. Take $x_0\in\Omega$ and a test function $\phi\in C^2(\Omega)$ with $\nabla_{\mathbb{H}}\phi(x_0)\neq 0$ that touches $\overline{u}$ at $x_0$ from above, i.e. $
	\overline{u}(x_0)=\phi(x_0)$ and $\phi-\overline{u}$ has a strict minimum at $x_0$.
	By definition of $\overline{u}$ there exist $\eps_n\to 0^+$ and $x_n\to x_0$ such that
	\begin{equation}\label{sequence}
	u_{\eps_n}(x_n)\longrightarrow \overline{u}(x_0).
	\end{equation}
	Fix $B^\mathbb{H}_r(x_0)\subset\Omega$ so that $\nabla_{\mathbb{H}}\phi\neq 0$ on this set.
	By definition of infimum, for every $n\in\mathbb{N}$ there exists $\overline{x}_{n}\in B^\mathbb{H}_r(x_0)$ such that
	\begin{equation}\label{almostMinimizing}
	\phi(\overline{x}_n)-u_{\eps_n}(\overline{x}_n) 
	\leq \inf_{B^\mathbb{H}_r(x_0)}(\phi - u_{\eps_n})+\eps_n^3.
	\end{equation}
	
	By possibly passing to a subsequence, we can assume $\overline{x}_n\longrightarrow \overline{x}_0$ for some $\overline{x}_0\in \overline{B^\mathbb{H}_r(x_0)}$.
	We get
	\begin{equation*}
	\begin{split}
	\phi(\overline{x}_0)-\overline{u}(\overline{x}_0)
	&=\liminf_{\substack{\eps\to 0^+\\ q\to \overline{x}_0 }}(\phi(q)-u_\eps(q))\\
	&\leq \liminf_{n\to\infty} \left(\phi(\overline{x}_n)-u_{\eps_n}(\overline{x}_n)\right)\\
	&\leq \liminf_{n\to\infty}\left(\inf_{B^\mathbb{H}_r(x_0)}(\phi-u_{\eps_n})+\eps_n^3\right)\\
	&\leq \liminf_{n\to\infty} \left(\phi(x_n)-u_{\eps_n}(x_n)+\eps_n^3\right)\\
	&=\phi(x_0)-\overline{u}(x_0),
	\end{split}
	\end{equation*}
	where we used \eqref{sequence} and \eqref{almostMinimizing} respectively in the third and second lines.
	As a consequence $\overline{x}_0=x_0$, because $\phi-\overline{u}$ has a strict minimum at $x_0$. 
	Note that inequality \eqref{almostMinimizing} rewrites as
	$$ u_{\eps_n}(x)-u_{\eps_n}(\overline{x}_n)\leq \phi(x)-\phi(\overline{x}_n)+\eps_n^3 \quad \text{for all}\;\; x\in B^\mathbb{H}_r(x_0).$$
	For $n$ large enough, we have $\eps_n<r/2$ and $\overline{x}_n\in B^\mathbb{H}_{r/2}(x_0)$ so that
	$B^\mathbb{H}_{\eps_n}(\overline{x}_n)\subseteq B^\mathbb{H}_r(x_0)$. Therefore
	\begin{equation*}
	\begin{split}
	\mu_p(u_{\eps_n}, \eps_n)(\overline{x}_n)-u_{\eps_n}(\overline{x}_n)
	&\leq \mu_p(\phi,\eps_n)(\overline{x}_n)-\phi(\overline{x}_n)+\eps_n^3\\
	&=c\eps_n^2 \Delta^N_{\mathbb{H},p}  \phi(\overline{x}_n)
	+o(\eps_n^2),
	\end{split}
	\end{equation*}
where we used Lemma \ref{amvph}. We remark that the proof of this Lemma shows that the error in the above expansion is uniform on the set where $\phi$ has nonvanishing horizontal gradient. Since $\nabla_{\mathbb{H}}\phi\neq 0$ on $B^\mathbb{H}_r(x_0)$, such error is independent of $x_n$.
	By definition of $u_{\eps_n}$ we get
	$$ 0\leq \Delta^N_{\mathbb{H},p}  \phi(\overline{x}_n)
	+o(1), $$
	which, after taking the limit as $n\to\infty$, concludes the proof for $\overline{u}$.
\end{proof}
We finalize now the proof of Theorem \ref{maintheorem}.

\begin{proof}
By Theorem \ref{TheoremBoundary}, for $y\in\partial\Omega$ we have
\begin{equation*}
\begin{split}
\limsup_{\Omega\ni x\to y}\overline{u}(x)\leq G(y)\leq \liminf_{\Omega\ni x\to y}\underline{u}(x).
\end{split}
\end{equation*} 
Using the comparison principle for viscosity solutions of the $p$-Laplace equation in $\mathbb{H}$ (see \cite{Bieske}),  we get $\overline{u}\leq \underline{u}$ in $\Omega$. 
 Since trivially  $\overline{u}\geq \underline{u}$, we have that $\overline{u}= \underline{u}:=u$ is a viscosity solution of $-\Delta^N_{\mathbb{H},p} u = 0$ in $\Omega$. Moreover, $u$ attains the appropriate boundary value and therefore, by a compactness argument, Theorem  \ref{maintheorem} follows.
 \end{proof}

\bibliographystyle{alpha}
\bibliography{DMRSDec2020}

\end{document}